\documentclass[10pt]{amsart}
\usepackage{amsfonts,amssymb,amscd,amsmath,enumerate,verbatim,calc,amsthm}
\input xy
\xyoption{all}

\headsep=1cm \numberwithin{equation}{section}
\newtheorem{thm}{Theorem}[section]
\newtheorem{cor}[thm]{Corollary}
\newtheorem{lem}[thm]{Lemma}
\newtheorem{prop}[thm]{Proposition}
\newtheorem{defn}[thm]{Definition}

\newtheorem{rem}[thm]{Remark}

\newcommand{\coker}{\mbox{Coker}\,}
\newcommand{\Hom}{\mbox{Hom}\,}
\newcommand{\Ext}{\mbox{Ext}\,}
\newcommand{\Tor}{\mbox{Tor}\,}
\newcommand{\Spec}{\mbox{Spec}\,}
\newcommand{\Max}{\mbox{Max}\,}

\newcommand{\Ass}{\mbox{Ass}\,}

\newcommand{\Supp}{\mbox{Supp}\,}

\newcommand{\depth}{\mbox{depth}\,}
\renewcommand{\dim}{\mbox{dim}\,}

\newcommand{\pd}{\mbox{pd}\,}
\newcommand{\id}{\mbox{id}\,}
\newcommand{\fd}{\mbox{fd}\,}

\newcommand{\Id}{\mbox{Id}\,}

\renewcommand{\H}{\mbox{H}}

\newcommand{\fa}{\mathfrak{a}}

\newcommand{\fm}{\mathfrak{m}}
\newcommand{\fp}{\mathfrak{p}}
\newcommand{\fq}{\mathfrak{q}}

\begin{document}
\bibliographystyle{amsplain}


\title[Tensor product of $C$-injective modules]
 {Tensor product of $C$-injective modules}

\bibliographystyle{amsplain}

     \author[M. Rahmani]{Mohammad Rahmani}
     \author[A.- J. Taherizadeh]{Abdoljavad Taherizadeh}

\address{Faculty of Mathematical Sciences and Computer,
Kharazmi University, Tehran, Iran.}

\email{m.rahmani.math@gmail.com}
\email{taheri@khu.ac.ir}



\begin{abstract}
Let $R$ be a Noetherian ring and let $C$ be a semidualizing $R$-module. In this paper, we are concerned with the tensor and torsion product of
$C$-injective modules. Firstly, it is shown that the tensor product of any two $C$-injective $R$-modules is $C$-injective if and only if the
injective hull of $C$ is $C$-flat. Secondly, it is proved that $C$ is a pointwise dualizing $R$-module if and only if $\Tor^R_i(M,N)$ is $C$-injective for all $C$-injective $R$-modules $M$ and $N$, and all $ i \geq 0$. These results recover the celebrated theorems of Enochs and Jenda \cite{EJ2}.
\end{abstract}

\maketitle

\bibliographystyle{amsplain}
\section{introduction}

Throughout this paper, $R$ is a commutative Noetherian ring with non-zero identity. A classical question of Yoneda asks when the tensor
product of two injective modules is injective. The first partial result related
to this question was given by Hattori \cite{H2}. He showed that the tensor product
of two injective $R$-modules is injective when $R$ is an integral domain. It is easy to see that
the Hattori's result is also true when $R$ is the product of a
finite number of integral domains. Next, Ishikawa \cite{I} showed that if $R$ is
a Noetherian ring such that the injective hull $E(R)$ of $R$ is a flat $R$-module, then the
tensor product of any two injective $R$-modules is injective. It is well-known
that if $R$ is Gorenstein, then the injective hull of $R$ is flat. More precisely, if $R$ is Gorenstein, then $E(R) = S^{-1}R$, where $S$ is the set of non-zero divisors of $R$. As a generalization, Cheatham and Enochs in \cite{CE}, considered a ring $R$, with the property that $E(R)$ is flat, and showed that $R$ has this property if and only if $R_{\fp}$ is Gorenstein for all $\fp \in \Ass(R)$. Finally, a complete answer
to the question of Yoneda was given by Enochs and Jenda in \cite{EJ2}. They proved that $E(R)$ is flat if and only if the tensor
product of any two injective $R$-modules is injective. Also, they showed that $R$ is Gorenstein if and only if the torsion product of any two injective $R$-modules is injective.

Grothendieck \cite{H1} introduced dualizing modules as tools for investigating cohomology
theories in algebraic geometry. An $R$-module is called a pointwise dualizing, if its localization at any maximal ideal of $R$ is dualizing. Note that
 if $\dim(R) < \infty$, then a pointwise dualizing module is dualizing. Over
a Noetherian ring $R$, a finitely generated $R$-module $C$ is semidualizing if the natural
homothety map $ R\longrightarrow \Hom_R(C,C) $ is an isomorphism and $ \Ext^i_R(C,C)=0 $ for all $ i>0 $. Foxby \cite{F}, Vasconcelos \cite{V}
and Golod \cite{G} independently initiated the study of semidualizing modules.

In this paper, we generalize the main results of \cite{EJ2}, and characterize dualizing modules in terms of the torsion product of $C$-injective modules (See Definition 2.4). Note that if $C$ is dualizing, then $E(C) = S^{-1}C \cong C \otimes_R S^{-1}R $ is $C$-flat. Hence we consider those semidualizing $R$-modules $C$, whose
 injective hulls are $C$-flat. We prove that $E(C)$ is $C$-flat if and only if the tensor product of any two $C$-injective $R$-modules is $C$-injective. Next we characterize dualizing modules as follows: $C$ is pointwise dualizing if and only if the torsion product of any two $C$-injective $R$-modules is $C$-injective.

\section{preliminaries}

In this section, we recall some definitions and facts which are needed throughout this
paper. By an injective cogenerator, we always mean an injective $R$-module $E$ for which $ \Hom_R(M,E) \neq 0 $ whenever $M$ is a nonzero $R$-module. For an $R$-module $M$, the injective hull of $M$, is always denoted by $E(M)$. Also, the flat cover of $M$ is always denoted by $F(M)$. For basic definitions and properties of envelope and covers, see the textbook \cite{EJ1}.

\begin{defn}
  \emph{Let $\mathcal{X} $ be a class of $R$-modules and $M$ an $R$-module. An $\mathcal{X}$-\textit{resolution} of $M$ is a complex of $R$-modules in $\mathcal{X} $ of the form \\
     \centerline{ $X = \ldots \longrightarrow X_n \overset{\partial_n^X} \longrightarrow X_{n-1} \longrightarrow \ldots \longrightarrow X_1 \overset{\partial_1^X}\longrightarrow X_0 \longrightarrow 0$}
such that $\H_0(X) \cong M$ and $\H_n(X) = 0$ for all $ n \geq 1$.}
\emph{Also the $ \mathcal{X}$-\textit{projective dimension} of $M$ is the quantity \\
\centerline{ $ \mathcal{X}$-$\pd_R(M) := \inf \{ \sup \{ n \geq 0 | X_n \neq 0 \} \mid X$ is an $\mathcal{X}$-resolution of $M \}$}.}
\emph{So that in particular $\mathcal{X}$-$\pd_R(0)= - \infty $. The modules of $\mathcal{X}$-projective dimension zero are precisely the non-zero modules in $\mathcal{X}$. The terms of $\mathcal{X}$-\textit{coresolution} and $\mathcal{X}$-$\id$ are defined dually.}
\end{defn}

The following remark will be useful in the proof of our main theorems. For the proof, see \cite[Theorems 3.2.12 and 3.2.14]{EJ1}.

\begin{rem}
\emph{Let $S$ be an $R$-algebra. Let $M$ be an $R$-module and let $N$ and $L$ be $S$-modules. The \textit{Hom-evaluation} homomorphism,}

\centerline{ $\theta_{LNM} : L \otimes_S \emph\Hom_R(N,M) \rightarrow \emph\Hom_R(\emph\Hom_S(L,N),M)$}
\emph{given by $ \theta_{LNM}(l \otimes \psi)(\varphi) = \psi (\varphi (l))$ where $l \in L$, $\psi \in \Hom_R(N,M)$ and $\varphi \in \Hom_S(L,N)$, is an isomorphism in either of the following conditions:}
\begin{itemize}
           \item[(i)]{$L$ \emph{is finitely generated and projective.}}
            \item[(ii)]{$L$ \emph{is finitely generated, and $M$ is injective.}}
\end{itemize}
\emph{Also the \textit{tensor-evaluation} homomorphism,}

\centerline{$ \omega_{LNM} : \emph\Hom_S(L,N) \otimes_R M \rightarrow \emph\Hom_S(L, N \otimes_R M)$ }
\emph{given by $ \omega_{LNM} (\psi \otimes m)(l) = \psi(l) \otimes m$ where $l \in L$, $m \in M$ and $\psi \in \Hom_S(L,N)$, is an isomorphism in either of the following conditions:}
\begin{itemize}
           \item[(i)]{$L$ \emph{is finitely generated and projective.}}
            \item[(ii)]{$L$ \emph{is finitely generated, and $M$ is flat.}}
\end{itemize}
\end{rem}

\begin{defn}
 \emph{A finitely generated $ R $-module $ C $ is \textit{semidualizing} if it satisfies the following conditions:
\begin{itemize}
             \item[(i)]{The natural homothety map $ R\longrightarrow \Hom_R(C,C) $ is an isomorphism.}
             \item[(ii)]{$ \Ext^i_R(C,C)=0 $ for all $ i>0 $.}
          \end{itemize}}
\end{defn}
For example a finitely generated projective $R$-module of rank 1 is semidualizing. If $R$ is Cohen-Macaulay, then an $R$-module $D$ is dualizing if it is semidualizing and that $\id_R (D) < \infty $ . For example the canonical module of a Cohen-Macaulay local ring, if exists, is dualizing.

\begin{defn}
\emph{Following \cite{HJ}, let $C$ be a semidualizing $R$-module. We set \\
 $\mathcal{F}_C(R) =$ the subcategory of $R$-modules  $C \otimes_R F$ where $F$ is a flat $R$--module. \\
 $\mathcal{I}_C(R) =$ the subcategory of $R$-modules  $\Hom_R(C,I) $ where $I$ is an injective $R$--module. \\
Modules in $\mathcal{F}_C(R)$ (resp. $\mathcal{I}_C(R)$) are called $C$-\textit{flat} (resp. $C$-\textit{injective}).
We use the notation $C$-$\fd$ (resp. $C$-$\id$)  instead of  $\mathcal{F}_C$-$\pd$ (resp. $\mathcal{I}_C$-$\id$ ).
A complete $ \mathcal{F} \mathcal{F}_C$-\textit{resolution} is a complex $X$ of $R$-modules such that
\begin{itemize}
           \item[(i)]{$X$ is exact and $X \otimes_R I$ is exact for each $I \in \mathcal{I}_C(R)$, and that}
            \item[(ii)]{$X_i \in \mathcal{F}_C(R)$ for all $i < 0$ and $X_i$ is flat for all $i \geq 0$.}
             \end{itemize}
             An $R$-module $M$ is called G$_C$-\textit{flat} if there exists a complete $ \mathcal{F} \mathcal{F}_C$-resolution $X$ such that
$M \cong \coker (\partial_1^X)$. All flat $R$-modules and all $R$-modules in $\mathcal{F}_C(R)$ are G$_C$-flat.}

\emph{A complete $\mathcal{I}_C \mathcal{I}$-{coresolution} is a complex $Y$ of $R$-modules such that
	\begin{itemize}
		\item[(i)]{$Y$ is exact and $\Hom_R(I,Y)$ is exact for each $I \in \mathcal{I}_C(R)$, and that}
		\item[(ii)]{$Y^i \in \mathcal{I}_C(R)$ for all $i \geq 0$ and $Y^i$ is injective for all $i < 0$.}
	\end{itemize}
	An $R$-module $M$ is called G$_C$-\textit{injective} if there exists a complete $\mathcal{I}_C \mathcal{I}$-coresolution $Y$ such that
	$M \cong \coker (\partial^1_Y)$. All injective $R$-modules and all $R$-modules in $\mathcal{I}_C(R)$ are G$_C$-injective. Note that when $C = R$ these notions recover the concepts of Gorenstein flat modules
	and Gorenstein injective modules, respectively.}
\end{defn}

In the rest of this section, we collect the basic properties of (semi)dualizing modules and related homological dimensions.

\begin{prop}\label{B0}
Let $C$ be a semidualizing $R$-module. Then we have the following:
\begin{itemize}
           \item[(i)]{$\emph\Supp (C) = \emph\Spec (R)$, $\emph\dim C = \emph\dim R$ and $\emph\Ass (C) = \emph\Ass (R)$.}
           \item[(ii)] { If $R$ is local, then $C$ is indecomposable. }
            \item[(iii)]{ If $ x \in R $ is $R$--regular, then $C/ xC$ is a semidualizing $R/ xR$-module .}
             \item[(iv)]{$\emph\depth_R (C) = \emph\depth (R) $.}
             \item[(v)]{ If $R \rightarrow S$ is a flat ring homomorphism, then $ C \otimes_R S$ is a semidualizing $S$-module.}
             \item[(vi)]{ Let $M$ be an $R$-module. Then $M \neq 0$ if and only if $ \emph\Hom_R(C,M) \neq 0$ . Also $M \neq 0$ if and only if  $ C \otimes_R M \neq 0 $ .}
                 \end{itemize}
\end{prop}
\begin{proof}

 The parts (i), (iii) and (v) follow from the definition of semidualizing modules. For (ii), note that if $R$ is local, then $R$ is indecomposable as an $R$-module. Now the isomorphism $R \cong  \Hom_R (C,C) $ shows that $C$ is indecomposable. For (iv), note that an element of $R$ is $R$-regular if and only if it is $C$-regular since $\Ass (C) = \Ass (R)$. Now an easy induction yields the equality. For (vi), note that $M \neq 0$ if and only if $\Ass_R (\Hom_R(C,M)) = \Ass_R(M) \neq \phi $, and this is the case if and only if $\Hom_R(C,M) \neq 0$. Finally if $E$ is an injective cogenerator, then  $M \neq 0$ if and only if $\Hom_R(M,E) \neq 0$ if and only if
$ \Hom_R(C \otimes_R M,E) \cong \Hom_R(C,\Hom_R(M,E)) \neq 0$, and this is the case if and only if $C \otimes_R M \neq 0$.
\end{proof}

\begin{prop}\label{A1}
Let $C$ be a semidualizing $R$-module, $M$ an $R$-module and $E$ be an injective $R$-module.
\begin{itemize}
           \item[(i)]{ If $M$ is a $C$-flat $R$-module, then $\emph\Hom_R(M,E)$ is a $C$-injective $R$-module. Also, if $(R, \fm)$ is local, $M$ is Matlis reflexive and $\emph\Hom_R(M,E(R/ \fm))$ is $C$-flat, then $M$ is $C$-injective. }
           \item[(ii)]{If $M$ is a $C$-injective $R$-module, then $\emph\Hom_R(M,E)$ is a $C$-flat $R$-module }.
           \item[(iii)]{Let $R \rightarrow S$ be a flat ring homomorphism. Then $(C \otimes_R S)$-$\emph\fd_S( M \otimes_R S) \leq C$-$\emph\fd_R(M)$. Also the equality holds if $S$ is a faithfully flat $R$-module }.
               \end{itemize}
\end{prop}
\begin{proof}
 For (i), use the Hom-tensor adjointness isomorphism. For (ii), use the Hom-evaluation isomorphism and the fact that $\Hom_R(I,E)$ is flat whenever $I$ is injective. For (iii), observe that if $F$ (resp. $E$) is a flat (resp. an injective) $R$-module, then $F \otimes_R S$ (resp. $E \otimes_R S$) is a flat (resp. an injective) $S$-module.
\end{proof}

\begin{prop}\label{A1}
Let $C$ be a semidualizing $R$-module, $M$ an $R$-module and $x \in R$ is $R$-regular.
\begin{itemize}
           \item[(i)]{If $M$ is a $C$-flat $R$-module, then $x$ is $M$-regular. Also $M/xM$ is a $C/xC$-flat $R/xR$-module }.
           \item[(ii)]{If $M$ is a $C$-injective $R$-module, then $\emph\Hom_R(R/xR , M)$ is a $C/xC$-injective $R/xR$-module}.
                 \end{itemize}
\end{prop}
\begin{proof}
It is straightforward.
\end{proof}

\begin{prop}\label{A1}
Let $(R, \fm)$ be a local ring and let $C$ be a semidualizing $R$-module.
\begin{itemize}
           \item[(i)]{$C$ is a dualizing $R$-module if and only if $C \otimes_R \widehat{R}$ is a dualizing $\widehat{R}$-module }.
           \item[(ii)]{ Let $x \in \fm$  be $R$-regular. Then $C$ is a dualizing $R$-module if and only if $C/xC$ is a dualizing $R/xR$-module}.
                 \end{itemize}
\end{prop}
\begin{proof}
Just use the definition of dualizing modules.
\end{proof}

\begin{thm}\label{SQ}
Let $C$ be a semidualizing $R$-module and let $M$ be an $R$-module.

\begin{itemize}
           \item[(i)]{$C$-$\emph{\id}_R(M) = \emph{\id}_R (C \otimes_R M) $ and $\emph{\id}_R(M) =C$-$\emph{\id}_R(\emph{\Hom}_R(C,M))$}.
           \item[(ii)]{$C$-$\emph{\fd}_R(M) = \emph{\fd}_R (\emph{\Hom}_R(C,M))$ and $\emph{\fd}_R(M) =C$-$\emph{\fd}_R(C \otimes_R M) $}.
                 \end{itemize}
\end{thm}

\begin{proof}
For (i), see \cite[Theorem 2.11]{TW} and for (ii), see \cite[Proposition 5.2]{STWY}.
\end{proof}

\begin{defn}\label{A1}
\emph{An $R$-module $M$ is called a $(C,\fp)$-\textit{injective} $R$-module, if it is isomorphic to an $R$-module of the form $\Hom_R(C , E(R/ \fp))$, where $\fp$ is a prime ideal of $R$.}
\end{defn}

It is well-known that over a Noetherian ring, any injective module is isomorphic to a direct sum of indecomposable injective modules. It follows that, if $R$ is Noetherian, then any $C$-injective $R$-module $M$ is isomorphic to a direct sum of $(C,\fp)$-injective $R$-modules where $\fp \in \Ass_R(M)$. Note that, in this case, any direct sum of $C$-injective $R$-modules is again $C$-injective.

\begin{rem}\label{A1}
\emph{Let $M$ be a finitely generated $R$-module. There are isomorphisms\\
 \[\begin{array}{rl}
  \Hom_R(M,E(R/ \fp)) &\cong \Hom_R(M,E(R/ \fp) \otimes_R R_{\fp})\\
  &\cong \Hom_R(M,E(R/ \fp)) \otimes_R R_{\fp}\\
  &\cong \Hom_{R_{\fp}} (M_{\fp} , E_{R_{\fp}}(R_{\fp}/ \fp R_{\fp})),\\
  \end{array}\]
where the the first isomorphism holds because $E(R/ \fp) \cong E_{R_{\fp}}(R_{\fp}/ \fp R_{\fp})$, and the second isomorphism is tensor-evaluation. In particular, $\Hom_R(M,E(R/ \fp))$ is an Artinian $R_{\fp}$-module. Also, if we set $N= \Hom_R(M,E(R/ \fp))$, then the isomorphism $N \cong N_{\fp}$  yields another isomorphism\\
\centerline{$\Tor^R_i(N , N) \cong \Tor^{R_{\fp}}_i(N_{\fp} , N_{\fp}), $}
for all $ i \geq 0$.}
\end{rem}

\section{main results}

Throughout this section, $C$ is a semidualizing $R$-module. We begin with two lemmas.
\begin{lem}\label{A2}
Let $\fp$ be a prime ideal of $R$ and let $M$ be a $(C, \fp)$-injective $R$-module. Then $M \otimes_R M \neq 0$ if and only if $\emph\depth_{R_{\fp}} (C_{\fp}) = 0$.
\end{lem}
\begin{proof}
First assume that $\depth_{R_{\fp}} (C_{\fp}) = 0$. Then $\depth (R_{\fp}) = 0 $ and so by \cite[Lemma 2.2]{EJ2}, we have $ E(R/ \fp) \otimes_R E(R/ \fp) \neq 0 $. Now using the Hom-evaluation isomorphism and Proposition 2.5(vi), one can see that $M \otimes_R M \neq 0$. Conversely, if $M \otimes_R M \neq 0$ then another use of Hom-evaluation isomorphism and Proposition 2.5(vi), show that $ E(R/ \fp) \otimes_R E(R/ \fp) \neq 0 $. Therefore in view of \cite[Lemma 2.2]{EJ2}, we obtain $\depth_{R_{\fp}} (C_{\fp}) = \depth R_{\fp} = 0$, as wanted.
\end{proof}

In \cite{K}, Kubik introduced the dual notion of semidualizing modules. Over a Noetherian local ring $(R, \fm)$, an Artinian $R$-module $T$ is called a quasidualizing $R$-module, if the natural homothety map $\widehat{R} \rightarrow \Hom_R(T,T)$ is an isomorphism and $ \Ext^i_R(T,T)=0 $ for all $ i>0 $. She proved that a quasidualizing module is a cogenerator, that is, if $\Hom_R(L,T) = 0$, for an $R$-module $L$, then $L=0$. She then asked \cite[Question 3.12]{K}, does $ T \otimes_R L = 0$ imply $L=0$? Lemma 3.1 provides a counterexample. Let $(R, \fm)$ be complete local with $\depth(R) > 0$. Then one checks easily that $\Hom_R(C, E(R/ \fm))$ is quasidualizing. Now by lemma 3.1, we have $\Hom_R(C, E(R/ \fm)) \otimes_R \Hom_R(C, E(R/ \fm)) = 0$.


\begin{lem}\label{A2}
Let $\fp$ be a prime ideal of $R$ and $M$ be a $(C,\fp)$-injective $R$-module. Then $M \otimes_R M $ is a non-zero $C$-injective $R$-module if and only if
$ C_{\fp} $ is an injective $ R_{\fp} $-module.
\end{lem}
\begin{proof}
First assume that $ C_{\fp} $ is an injective $ R_{\fp} $-module. Then we have $ \dim (R_{\fp}) = \dim_{R_{\fp}} (C_{\fp}) = 0$ and hence $R_{\fp}$ is Artinian. Now lemma 3.1 shows that $M \otimes_R M \neq 0$. Next we can write $ C_{\fp} \cong E_{R_{\fp}}(R_{\fp}/ \fp R_{\fp})^{\mu^0( \fp , C)}$ . But Proposition 2.5(v) shows that $ C_{\fp}$ is semidualizing for $ R_{\fp}$ and then by Proposition 2.5(ii) we must have $\mu^0( \fp , C) = 1$, that is, $ C_{\fp} \cong E_{R_{\fp}}(R_{\fp}/ \fp R_{\fp})$. By Remark 2.11, we have the $R_{\fp}$-module isomorphisms \\
\centerline{$ \Hom_R(C,E(R/ \fp)) \cong \Hom_{R_{\fp}} (C_{\fp} , E_{R_{\fp}}(R_{\fp}/ \fp R_{\fp}) \cong \widehat{R_{\fp}} = R_{\fp} $,}
and the later is a flat $R$-module. Therefore using the tensor-evaluation isomorphism, we have the following isomorphisms

 \[\begin{array}{rl}
  M \otimes_R M &\cong \Hom_R(C,E(R/ \fp)) \otimes_R \Hom_R(C,E(R/ \fp))\\
  &\cong\Hom_R(C,E(R/ \fp)) \otimes_R R_{\fp}\\
  &\cong\Hom_R(C,E(R/ \fp)\otimes_R R_{\fp}),\\
  \end{array}\]
and the later module is $C$-injective.\\

 Conversely, assume that $M \otimes_R M $ is a non-zero $C$-injective $R$-module.
 Then $C \otimes_R ( M \otimes_R M) \cong E(R/ \fp) \otimes_R \Hom_R(C,E(R/ \fp))$ is a non-zero injective $R$-module by Theorem 2.9(i). On the other hand, by Remark 2.11, $\Hom_R(C,E(R/ \fp))$ is an Artinian $R_{\fp}$-module. Therefore $ C  \otimes_R M \otimes_R M $ is of finite length by \cite[Corollary 7.4]{KLW}. Thus by using the tensor-evaluation isomorphism, one can see that $E(R/ \fp) \otimes_R \Hom_R(C,E(R/ \fp))$ is a finite length injective $R_{\fp}$-module and so  \\
 \centerline{$\Hom_{R_{\fp}}(E(R/ \fp) \otimes_R \Hom_R(C,E(R/ \fp)) , E(R/ \fp)) \cong \Hom_{R_{\fp}}(\Hom_R(C,E(R/ \fp)), \widehat{R_{\fp}} ) $}
is a finite length flat $R_{\fp}$-module. It follows that $ \dim (R_{\fp}) = 0$ and then $ \widehat{R_{\fp}} = R_{\fp}$ is Artinian. Therefore, we can replace $R$ by $R_{\fp}$ and assume that $ (R, \fm)$ is an Artinian local ring. We want to show that $C$ is an injective $R$-module.
  Set $ (-)^{\vee} = \Hom_R(-, E(R/ \fm))$. Let $n$ be a positive integer for which  $ \fm^n = 0 $ and $ \fm^{n-1} \neq 0$. Note that, with the new notations, 
  $\Hom_R( C^{\vee} , R)$ is a flat $R$-module of finite length. Hence, $\Hom_R( C^{\vee} , R)$  must be free, and therefore
   $ \fm^{n-1} \Hom_R( C^{\vee} , R) \neq 0$. Now if $ \theta (C^{\vee}) \subseteq \fm$ for
   each $ \theta \in \Hom_R( C^{\vee} , R)$, then $ \fm^{n-1} \Hom_R( C^{\vee} , R) = 0$ which is impossible. Consequently, there exists an element
   $ \theta \in \Hom_R( C^{\vee} , R)$ such that $ \theta (C^{\vee}) \nsubseteq \fm $. Hence $ \theta $ is onto and so is split. Thus we conclude that
   $R$ is a direct summand of $ C^{\vee} $. Write $ C^{\vee} = R \oplus K $. Taking Matlis dual, yields $ C \cong E(R / \fm) \oplus K^{\vee}$, and the indecomposablity of $C$ implies that $ K^{\vee} = 0$, whence $ K=0$ and $ C \cong E(R/ \fm)$. Hence $C$ is injective, as wanted.
\end{proof}
Now we are ready to prove one of our main results. The following theorem generalizes \cite[Theorem 3]{CE}.

\begin{thm}\label{SQ}
The following are equivalent:

\begin{itemize}
           \item[(i)]{$ E(C)$ is $C$-flat.}
           \item[(ii)] { $C$-$\emph{\fd}_R (E(C)) < \infty $. }
            \item[(iii)]{ $ C_{\fp}$ is an injective $ R_{\fp}$--module for all $ \fp \in \emph{Ass}_R(C) $.}
             \item[(iv)]{ $ E(M)$ is $C$-flat for all $C$-flat $R$-modules $M$.}
             \item[(v)]{ $ E(M)$ is $C$-flat for all \emph{G}$_C$-flat $R$-modules $M$.}
             \item[(vi)]{ $F(M)$ is $C$-injective for all $C$-injective $R$-modules $M$.}
             \item[(vii)]{ $F(M)$ is $C$-injective for all \emph{G}$_C$-injective $R$-modules $M$.}
             \item[(viii)]{ $ N \otimes_R N^{\prime}$ is $C$-injective for all $C$-injective $R$-modules $N$ and $N^{\prime}$.}
             \item[(ix)]{ $ S^{-1} C$ is an injective $R$-module, where $S= \emph{Nzd}_R(C) $.}
                 \end{itemize}
\end{thm}

\begin{proof}
   (i)$\Longrightarrow$(ii). It is evident.\

(ii)$\Longrightarrow$(iii). We can write $E(C) = \displaystyle  \oplus _{ \fp \in \emph{\Ass}_R(C)} E(R/ \fp)^{\mu^0( \fp , C)}$. Now, in view of Theorem 2.9(ii), we have
$ \fd_R (\Hom_R(C,E(C)) < \infty $, and thus \\
  \centerline{$ \fd_{R_{\fp}} (\Hom_{R_{\fp}}(C_{\fp} , E_{R_{\fp}}(R_{\fp}/ \fp R_{\fp})) = \fd_R(\Hom_R(C,E(R/ \fp))) < \infty,$}
  for all $ \fp \in \Ass_R(C)$.
Hence $ \id_{R_{\fp}}(C_{\fp}) < \infty $, and so we have $ \id_{R_{\fp}}(C_{\fp}) = \depth (R_{\fp}) = 0$ by \cite[Corollary 9.2.17]{EJ1}. Therefore $C_{\fp}$ is an injective $ R_{\fp}$-module for all $ \fp \in \Ass_R(C) $.

(i)$\Longrightarrow$(viii). Let $N$ and $N^{\prime}$ be $C$-injective $R$-modules. Let $\fp$ and $\fq$ be two distinct prime ideals of $R$. Then either
$ \fp \nsubseteq \fq$ or $ \fq \nsubseteq \fp$ . Suppose that $ \fp \nsubseteq \fq$. Set $M = \Hom_R(C,E(R/ \fp))$ and $N = \Hom_R(C,E(R/ \fq))$. Then we have
 \[\begin{array}{rl}
  M \otimes_R N &\cong M \otimes_{R_{\fp}} R_{\fp} \otimes_R N \\
  &\cong M \otimes_{R_{\fp}} N_{\fp}\\
  &= 0,\\
  \end{array}\]
since $ E(R/ \fq)_{\fp} = 0$. Next we prove
that if $ \fp \notin \Ass_R(C)$, then \\
 \centerline{ $ \Hom_R(C,E(R/ \fp)) \otimes_R \Hom_R(C,E(R/ \fp)) = 0$.}
 By Proposition 2.5(vi) and Remark 2.2, $ \Hom_R(C,E(R/ \fp)) \otimes_R \Hom_R(C,E(R/ \fp)) = 0$ if and only if 
 $ E(R/ \fp) \otimes_R \Hom_R(C,E(R/ \fp)) \cong C \otimes_R \Hom_R(C,E(R/ \fp)) \otimes_R \Hom_R(C,E(R/ \fp)) = 0 $, and 
 $ E(R/ \fp) \otimes_R \Hom_R(C,E(R/ \fp)) = 0 $ if and only if 
 $ E(R/ \fp) \otimes_R E(R/ \fp) \cong C \otimes_R E(R/ \fp) \otimes_R \Hom_R(C,E(R/ \fp)) = 0 $.
   Therefore, in view of Proposition 2.5(i), we need only to prove that if $ \fp \notin \Ass_R(R)$,
 then $ E(R/ \fp) \otimes_R E(R/ \fp) = 0$. We claim that if $\fp \notin \Ass(R)$, then there exists an element $ r \in \fp$ which is $R$-regular. For otherwise, the prime avoidance theorem implies that $ \fp \subseteq \fq$ for some  $\fq \in \Ass(R)$. But as we have seen in  (ii)$\Longrightarrow$(iii), in this case
 $ R_{\fq}$ is an Artinian ring and then  $ \fp = \fq \in \Ass(R) $, which is impossible.
 Next observe that the exact sequence $ 0 \rightarrow R \overset{r} \rightarrow R$ implies an
 epimorphism $ E(R/ \fp) \overset{r} \rightarrow E(R/ \fp) \rightarrow 0$, which is locally nilpotent. Now it is easy to
  see that $ E(R/ \fp) \otimes_R E(R/ \fp) = 0$. Therefore in order to prove that $ N \otimes_R N^{\prime}$ is $C$-injective, it is enough to
  prove that $ \Hom_R(C,E(R/ \fp)) \otimes_R \Hom_R(C,E(R/ \fp)) $ is $C$-injective for all $ \fp \in \Ass_R(C) $. The assumption together with Theorem 2.9(ii), show that
  $ \Hom_R(C,E(R/ \fp))$ is flat for all $ \fp \in \Ass_R(C) $. Hence $ E(R/ \fp) \otimes_R \Hom_R(C,E(R/ \fp))$ is injective and thus using the tensor-evaluation isomorphism, we see that\\
    \centerline{$ \Hom_R(C,E(R/ \fp)) \otimes_R \Hom_R(C,E(R/ \fp)) \cong \Hom_R\big(C,E(R/ \fp) \otimes_R \Hom_R(C,E(R/ \fp))\big) $}
is $C$-injective. Finally, Since any direct sum of  $C$-injective $R$-modules is $C$-injective, we conclude that $ N \otimes_R N^{\prime}$ is $C$-injective.

 (iv)$\Longrightarrow$(vi). Suppose that $I$ is an injective $R$-module, and $ \theta : F \rightarrow \Hom_R(C,I) $ is flat cover. Assume that 
 $ F' $ is a flat $R$-module. Consider the composition \\
 \centerline{$ \xi: \Hom_R(F',F) \overset{\omega^\ast_F}\longrightarrow \Hom_R(F' , \Hom_R(C , C \otimes_R F)) \overset{\cong}\longrightarrow \Hom_R(C \otimes_R F' , C \otimes_R F) $,}
in which $ \omega^\ast_F = \Hom_R(F', \omega_F) $ and $ \omega_F : F \rightarrow \Hom_R(C , C \otimes_R F) $ is the tensor-evaluation isomorphism. This isomorphism is such that $ \xi(\alpha) = \Id_C \otimes \alpha $ for all $ \alpha \in \Hom_R(F',F) $.  Now it is easy to check that $ \Id_C \otimes \theta : C \otimes_R F \rightarrow C \otimes_R \Hom_R(C,I) \cong I $ is $ \mathcal{F}_C $-cover of $I$. But since $I$ is injective, there exists a homomorphism $ E(C \otimes_R F) \rightarrow I $ making the
following diagram
  \begin{displaymath}
  \xymatrix{
  	0 \ar[r]	&  C \otimes_R F \ar@{^{(}->}[r] \ar[d]_{\emph{Id}_C \otimes \theta} &  E(C \otimes_R F) \ar@{-->}[dl] \\
  	& I  \\}
  \end{displaymath}
 commute. On the other hand, $ E(C \otimes_R F) $ is $C$-flat by assumption. Hence, by definition of cover, there exists a homomorphism
  $ E(C \otimes_R F) \rightarrow C \otimes_R F $ whose restriction to $ C \otimes_R F $ is an automorphism, and making the
  following diagram commute:
    \begin{displaymath}
    \xymatrix{
    	&  C \otimes_R F  \ar[d]   \\
        E(C \otimes_R F) \ar@{-->}[ur]  \ar[r] & I . \\}
    \end{displaymath}
   Therefore, $ C \otimes_R F $ is a direct summand of $ E(C \otimes_R F) $ and so is injective. Thus, by Theorem 2.9(i), $ F \cong \Hom_R(C , C \otimes_R F) $ is $C$-injective, as wanted.

(vi)$\Longrightarrow$(vii). Let $M$ be a G$ _C $-injective $R$-module. Then, by definition, $M$ is a homomorphic image of a $C$-injective $R$-module, say $L$. But $ F(L) $ is $C$-injective by assumption. Now by definition of flat cover, $ F(M) $ is a direct summand of $ F(L) $, and hence is $C$-injective.

(vii)$\Longrightarrow$(v). Let $M$ be a G$ _C $-flat $R$-module and let $E$ be an injective cogenerator. Let $ R \ltimes C $ denote the trivial extension of $R$ by $C$ and view M as an
$ R \ltimes C $-module via the natural surjection $ R \ltimes C \rightarrow R $. Now $M$ is G-flat
over $ R \ltimes C $ by \cite[Theorem 2.16]{HJ}. This is the case if and only if $ \Hom_R(M,E) $ is G-injective over $ R \ltimes C $. Another use of \cite[Theorem 2.16]{HJ} shows that $ \Hom_R(M,E) $ is G$ _C $-injective. Hence, by assumption, $ F = F(\Hom_R(M,E)) $ is $C$-injective. Set $ (-)^{\vee} = \Hom_R(-, E)$ and $ F = \Hom_R(C,I) $, where $I$ is injective. Then we have \\
\centerline{ $ M \hookrightarrow M^{\vee \vee} \hookrightarrow \Hom_R(C,I)^{\vee} \cong C \otimes_R I^{\vee} $,}
in which the isomorphism is Hom-evaluation. Finally, since $ C \otimes_R I^{\vee} $ is injective, $ E(M) $  must be a direct summand of $ C \otimes_R I^{\vee} $, and hence must be $C$-flat.

 (viii)$\Longrightarrow$(iii). Let $ \fp \in \Ass_R(C)$. Then $ \depth_{R_{\fp}}(C_{\fp}) = 0$ and hence $ \Hom_R(C,E(R/ \fp)) \otimes_R \Hom_R(C,E(R/ \fp))$ is a non-zero $C$-injective $R$-module by lemma 3.1 and the assumption. Thus $C_{\fp}$ is an injective $ R_{\fp}$-module by lemma 3.2 .

 (iii)$\Longrightarrow$(ix). Note that $\Spec(S^{-1}R) = \{ \fq S^{-1}R \mid \fq \in \Ass(R)\}$, and that $R_{\fq}$ is Artinian for all  $\fq \in \Ass(R)$.
 It, therefore, follows that $\Supp_{S^{-1}R} (\Ext_{S^{-1}R}^i(T,S^{-1}C)) = \phi$ for all $ i \geq 1$ and all finitely generated $S^{-1}R$-modules $T$,
 since $(S^{-1}C)_{\fq S^{-1}R} \cong C_{\fq}$ is an injective $((S^{-1}R)_{\fq S^{-1}R} \cong) R_{\fq}$-module. Thus $S^{-1}C$ is an
 injective $S^{-1}R$-module, and so is an injective $R$-module.

(ix)$\Longrightarrow$(iv). Let $M = C \otimes_R F$ be a $C$-flat $R$-module. As $F$ is a flat $R$-module, we can
   write $F =  \underset{ \underset {i \in I} \longrightarrow } \lim (F_i)   $, where $F_i$ is a finitely generated free $R$-module. Observe
    that $E(C) = S^{-1}C$, because $ C \hookrightarrow S^{-1}C$ is essential and that $S^{-1}C$ is an injective $R$-module by assumption. Now it is easy to see that $E(M) = E(C \otimes_R F) = C \otimes_R S^{-1}F $.

(iv)$\Longrightarrow$(v). Let $M$ be a G$_C$-flat $R$-module. By definition, $M$ can be embedded in a $C$-flat $R$-module, $N$ say. Now $E(M)$ is a direct summand of $E(N)$, and hence is $C$-flat since $E(N)$ is $C$-flat by assumption.

(v)$\Longrightarrow$(i). It is trivial, since $C$ itself is a G$_C$-flat $R$-module.
\end{proof}

\begin{rem}
\emph{ Let $M$ be a G$_C$-flat $R$-module. As shown in Theorem 3.3, if $E(C)$ is $C$-flat then $E(M)$ is $C$-flat. Set $E(M) = C \otimes_R F$, where $F$ is a flat $R$-module. Then $F \cong \Hom_R(C, E(M))$. On the other and, by the definition, $M$ is a submodule of a $C$-flat $R$-module which implies that $\Ass_R(M) \subseteq \Ass_R(C)$. But
 Theorem 3.3 implies that $C_{\fp} \cong E_{R_{\fp}}(R_{\fp}/ \fp R_{\fp})$ for any prime ideal $\fp \in \Ass_R(C)$. Hence we have}
  \[\begin{array}{rl}
  \ F &\cong \emph{\Hom}_R \big(C,  \bigoplus _{ \fp \in \emph{\Ass}_R(M)} E(R/ \fp)^{\mu^0(\fp , M)} \big)\\
  &\cong \bigoplus_{ \fp \in \emph{\Ass}_R(M)} \emph{\Hom}_{R_{\fp}} \big(C_{\fp} , E_{R_{\fp}}(R_{\fp}/ \fp R_{\fp})) \big)^{\mu^0(\fp , M)}\\
  &\cong \bigoplus_{ \fp \in \emph{\Ass}_R(M)} R_{\fp}^{\mu^0(\fp , M)}.\\
  \end{array}\]
\emph{ Note that, in this case, $R_{\fp}$ is Artinian and hence $ R_{\fp}^{\mu^0(\fp , M)}$ is a completion of a free $\widehat{R_{\fp}}$-module. Therefore $E(M) = \bigoplus_{ \fp \in \Ass_R(M)} (C \otimes_R T_{\fp})$, where $T_{\fp}$ is a completion of a free $\widehat{R_{\fp}}$-module.}
\end{rem}	

\begin{lem}\label{CC}
Assume that $ (R, \fm , k)$ is a Noetherian local ring. Then we have the following statements:
\begin{itemize}
           \item[(i)]{$\emph{\Ext}^i_R(\emph{\Hom}_R(C,E(k)),C \otimes_R \widehat{R}) \cong \emph{\Ext}^i_{\widehat{R}}(\emph{\Hom}_R(C,E(k)),C \otimes_R \widehat{R})$ for all $i \geq 0$.}
           \item[(ii)]{Suppose that $ x \in \fm$ is $R$- and  $ \emph{\Ext}_R^i(\emph{\Hom}_R(C,E(k)),C)$-regular for $ 1 \leq i \leq n$. Set $ \overline{(-)} = (-) \otimes_R R/xR$. Then for each $i$ with  $ 1 \leq i \leq n-1 $, there is an isomorphism
$ \emph{\Ext}_{\overline{R}}^{i-1}(\emph{\Hom}_{\overline{R}}({\overline{C}},E_{\overline{R}}(k)),{\overline{C}}) \cong \emph{\Ext}_R^i(\emph{\Hom}_R(C,E(k)),C) \otimes_R {\overline{R}}$
.}
\end{itemize}
\end{lem}

\begin{proof}
(i)  If $ 0 \rightarrow E^0(C) \rightarrow E^1(C) \rightarrow \cdots $ is a minimal injective
  resolution of $C$, then $ 0 \rightarrow E^0(C) \otimes_R \widehat{R} \rightarrow E^1(C) \otimes_R \widehat{R} \rightarrow \cdots $ is an injective resolution
  of $ C \otimes_R \widehat{R}$ as an $R$-module and $\widehat{R}$-module . Note that for each prime ideal $\fp$ of $R$, the injective $R$-module $ E(R/ \fp) \otimes_R \widehat{R}$, is a direct sum of copies of $E(R/ \fp)$. On the other hand, one can show that, by using the Hom-evaluation isomorphism, if $ \fp \neq \fm$,
  then
  \[\begin{array}{rl}
  \Hom_R(\Hom_R(C,E(k)) , E(R/ \fp)) &\cong  C \otimes_R \Hom_R(E(k) , E(R/ \fp))\\
    &= 0.\\
  \end{array}\]
Moreover, if $\widehat{\fq}$ is a prime ideal of $\widehat{R}$ with $ \widehat{\fq} \cap R \neq \fm$, then

  \[\begin{array}{rl}
  \Hom_R(\Hom_R(C,E(k)) , E_{\widehat{R}}(\widehat{R}/ \widehat{\fq})) &\cong  C \otimes_R \Hom_R(E(k) , E_{\widehat{R}}(\widehat{R}/ \widehat{\fq}))\\
    &= 0.\\
  \end{array}\]
Finally, by using the isomorphisms $ \widehat{R} \otimes_R E(k) \cong E(k)$, and $\Hom_R(E(k),E(k)) \cong \Hom_{\widehat{R}}(E(k),E(k))$, we have the desired isomorphism.

(ii) Set $ \overline{(-)} = (-) \otimes_R R/xR$. Note that $x$ is $C$-regular too. Let $ 0 \rightarrow E^0(C) \rightarrow E^1(C) \rightarrow ...$ be the minimal injective resolution of $C$. By applying the functor $ \Hom_R(\Hom_R(C,E(k)),-)$, we can compute $\Ext_R^i( \Hom_R(C,E(k)),C)$.
But since $\Hom_R( E(k),E(R/ \fp)) = 0$ for any prime $\fp \neq \fm$, one can see that
 \[\begin{array}{rl}
  \Hom_R(\Hom_R(C,E(k)),E^i(C)) &\cong C \otimes_R \Hom_R(E(k),E^i(C))\\
  &\cong C \otimes_R \Hom_R(E(k),E(k)^{\mu^i( \fm , C)})\\
  &\cong C \otimes_R \widehat{R}^{\mu^i( \fm , C)}.\\
  \end{array}\]
Hence we have $\Ext_R^i( \Hom_R(C,E(k)),C)$ as an $i$-th cohomology of the complex\\
 \centerline{ $ 0 \rightarrow C \otimes_R \widehat{R}^{\mu^0} \rightarrow C \otimes_R \widehat{R}^{\mu^1} \rightarrow ...$       ~~~~    (*)}
 where $ \mu^i = \mu^i( \fm , C)$.
Observe that $\mu^0 = 0$ because $\fm \notin \Ass_R(C)$. Note that by \cite[Lemma 9.2.2]{EJ1}, the complex\\
  \centerline{$ 0 \rightarrow \overline{C} \rightarrow \Hom_R( \overline{R} , E^1(C)) \rightarrow \Hom_R( \overline{R} , E^2(C)) \rightarrow  ...$ (**)}
is a minimal injective resolution of $\overline{C}$ as an $\overline{R}$-module. Thus, by applying the functor
  $\Hom_{\overline{R}}( \Hom_{\overline{R}}(\overline{C},E_{\overline{R}}(k)),-)$ on $ (**) $, we can compute
   $\Ext^i_{\overline{R}}( \Hom_{\overline{R}}(\overline{C},E_{\overline{R}}(k)),\overline{C})$ for all $i \geq 0$.
But then using the Hom-evaluation isomorphism and the fact that $ \Hom_R(\overline{R}, E(k)) \cong E_{\overline{R}}(k)$, we have the following
isomorphisms:
 \[\begin{array}{rl}
  \Hom_{\overline{R}}( \Hom_{\overline{R}}(\overline{C},E_{\overline{R}}(k)), \Hom_R(\overline{R}, E^i(C)))
   &\cong \overline{C} \otimes_{\overline{R}} \Hom_{\overline{R}}(E_{\overline{R}}(k), \Hom_R(\overline{R}, E^i(C)))\\
    &\cong \overline{C} \otimes_{\overline{R}} \Hom_R(\Hom_R(\overline{R}, E(k)) , E^i(C))\\
  &\cong \overline{C} \otimes_R \Hom_R(E(k), E^i(C)) \\
  &\cong \overline{C} \otimes_R \widehat{R}^{\mu^i} \\
  &\cong \overline{C} \otimes_{\overline{R}} \widehat{\overline{R}}^{\mu^i}.\\
  \end{array}\]
Let $B^i$ , $Z^i \subset C \otimes_R \widehat{R}^{\mu^i}$ be the images and kernels of the boundary maps of the complex $(*)$, respectively.
Observe that $x$ is not a zero-divisor on $ (C \otimes_R \widehat{R}^{\mu^i})/Z^i$
  since $ (C \otimes_R \widehat{R}^{\mu^i})/Z^i \hookrightarrow C \otimes_R \widehat{R}^{\mu^{i+1}}$ and $x$ is $(C \otimes_R \widehat{R}^{\mu^{i+1}})$-regular.
  Also the exact sequence\\
\centerline{$ 0 \rightarrow \Ext_R^i( \Hom_R(C,E(k)),C)= Z^i/B^i \rightarrow C \otimes_R \widehat{R}^{\mu^i})/B^i \rightarrow (C \otimes_R \widehat{R}^{\mu^i})/Z^i \rightarrow 0 $,}
in cojunction with the assumption, show that $x$ is $((C \otimes_R \widehat{R}^{\mu^i})/B^i)$-regular too. So that, we have two exact sequences\\
 \centerline{$ 0 = \Tor^R_1 (\overline{R} , (C \otimes_R \widehat{R}^{\mu^i})/B^i) \rightarrow B^i \otimes_R \overline{R} \rightarrow (C \otimes_R \widehat{R}^{\mu^i}) \otimes_R \overline{R} \cong  \overline{C} \otimes_{\overline{R}} \widehat{\overline{R}}^{\mu^i}$ ,}
 \centerline{$ 0 = \Tor^R_1 (\overline{R} , (C \otimes_R \widehat{R}^{\mu^i})/Z^i) \rightarrow Z^i \otimes_R \overline{R} \rightarrow (C \otimes_R \widehat{R}^{\mu^i}) \otimes_R \overline{R} \cong  \overline{C} \otimes_{\overline{R}} \widehat{\overline{R}}^{\mu^i}$.}
Consequently,  $B^i \otimes_R \overline{R}$ and $Z^i \otimes_R \overline{R}$ can be identified with images and kernels of the complex $(**)$. Therefore we have the isomorphism\\
 \centerline{$ \Ext_{\overline{R}}^{i-1}( \Hom_{\overline{R}}({\overline{C}},E_{\overline{R}}(k)),{\overline{C}}) \cong \Ext_R^i( \Hom_R(C,E(k)),C) \otimes_R {\overline{R}}$}
for all  $ 1 \leq i \leq n-1 $, as wanted.
\end{proof}

The following theorem is a generalization of \cite[Theorem 4.1]{EJ2}.
\begin{thm}\label{SQ}
The following are equivalent:
\begin{itemize}
           \item[(i)]{$C$ is pointwise dualizing.}
           \item[(ii)] { for any $\fp \in \emph\Spec(R)$, and any $(C ,\fp)$-injective $R$-module $M$, $ \emph\Tor^R_{\emph{ht}(\fp)}(M,M) \cong M $ and  $ \emph\Tor^R_i(M,M) = 0$ for all $i \neq \emph{ht}(\fp)$}.
           \item[(iii)]{ For all $C$-injective $R$-modules $N$ and $N^{\prime}$, and all $ i \geq 0$, $\emph\Tor^R_i(N,N^{\prime})$ is $C$-injective.}
          \end{itemize}
\end{thm}
 \begin{proof}
   (i)$\Longrightarrow$(ii). Let $\fp \in \Spec(R)$, and $ M = \Hom_R(C,E(R/ \fp))$. In view of Remark 2.11, there is a
  natural $R_{\fp}$-isomorphism $ \Tor^R_i(M,M) \cong  \Tor^{R_{\fp}}_i(M_{\fp},M_{\fp}) $ for each $ i \geq 0$. Therefore, in order to establish the desired isomorphism, it is enough to show that, if $(R, \fm , k)$ is a $d$-dimensional local ring and $ M = \Hom_R(C,E(k))$, then
  $ \Tor^R_d(M,M) \cong M $ and  $ \Tor^R_i(M,M) = 0$ for all $i \neq d$ .   Set $ (-)^{\vee} = \Hom_R(-, E(k))$. Then we have the isomorphisms\\
 \centerline{$ \Tor^R_i(M,M)^{\vee} \cong \Ext^i_R(M , M^{\vee}) \cong \Ext^i_R(M , C \otimes_R \widehat{R}).$}
 Observe that $ \mu^i(\fm , C) = 0$ for all $ i \neq d$, since $C$ is dualizing. Note also that $C \otimes_R \widehat{R}$ is dualizing for $\widehat{R}$ by Proposition 2.8(i). Now using lemma 3.5(i), we have \\
\centerline{$\Ext^i_R(M , C \otimes_R \widehat{R}) \cong \Ext^i_{\widehat{R}}(M , C \otimes_R \widehat{R}) = 0$,}
for all $i \neq d$, and that
     \[\begin{array}{rl}
  \Ext^d_R(M , C \otimes_R \widehat{R}) &\cong \Ext^d_{\widehat{R}}(M , C \otimes_R \widehat{R}) \\
  &\cong \Hom_{\widehat{R}}(M , E(k)^{\mu^d(\fm\widehat{R} , C \otimes_R \widehat{R})})\\
  &\cong \Hom_R(M , E(k)^{\mu^d(\fm , C)})\\
  &\cong \Hom_R(M ,E(k)), \\
    \end{array}\]
where the third and the fourth isomorphisms hold because both $M$ and $E(k)$ are Artinian and that $ \mu^i(\fm , C) = 1 = \mu^i(\fm\widehat{R} , C \otimes_R \widehat{R})$ for all
$ i \geq 0$.  It follows that $\Tor^R_i(M,M) = 0$ for all $ i \neq d$, and $ \Tor^R_d(M,M) \cong M $.

   (ii)$\Longrightarrow$(iii). Any $C$-injective $R$-module $M$, is a direct sum of $(C, \fp)$-injective $R$-modules, where $\fp$ runs over $\Ass_R(M)$. Now  just note that if $ \fp , \fq$ are two distinct prime ideals of $R$,
    then in view of Remark 2.11, we have $ \Tor^R_i(\Hom_R(C,E(R/ \fp)), \Hom_R(C,E(R/ \fq))) = 0$ for all $ i \geq 0$, and that any direct sum of $C$-injective $R$-modules is $C$-injective.

    (iii)$\Longrightarrow$(i). Since, by Remark 2.11, the isomorphism $ \Tor^R_i(M,M) \cong  \Tor^{R_{\fm}}_i(M_{\fm},M_{\fm}) $
  holds for all $(C, \fm)$-injective $R$-modules $M$, all $ i \geq 0$ and all $\fm \in \Max(R)$, we can replace $R$ by $R_{\fm}$ and assume that $(R, \fm , k)$ is local.
   Thus  $\Tor^R_i(\Hom_R(C,E(k)),\Hom_R(C,E(k))) $ is $C$-injective for all  $ i \geq 0$. Therefore by Proposition 2.6(ii), the $R$-module
    \[\begin{array}{rl}
  \Tor^R_i(\Hom_R(C,E(k)),\Hom_R(C,E(k)))^{\vee} &\cong \Ext^i_R(\Hom_R(C,E(k)) , \Hom_R(C,E(k))^{\vee}) \\
    &\cong \Ext^i_R(\Hom_R(C,E(k)),C \otimes_R \widehat{R})\\
  &\cong \Ext^i_{\widehat{R}}(\Hom_R(C,E(k)),C \otimes_R \widehat{R}), \\
    \end{array}\]
is $C$-flat. Hence in view of lemma 3.5(i) and Proposition 2.8(i), we can replace $R$ by $\widehat{R}$ and assume that $R$ is a complete local ring. Assume that $\depth(R) = d$, and that $\underline{x}:=x_1, \ldots , x_d$ is a maximal $R$-sequence in $\fm$. Then by Proposition 2.7(i) and Lemma 3.5(ii), we see that\\
  \centerline{$  \Ext_R^d( \Hom_R(C,E(k)),C) \otimes_R R/{\underline{x}}R \cong \Hom_{R/{\underline{x}}R}( \Hom_{R/{\underline{x}}R}(C/{\underline{x}}C,E_{R/{\underline{x}}R}(k)),C/{\underline{x}}C)$,}
  is a $C/{\underline{x}}C$-flat $R/{\underline{x}}R$-module.
   Therefore in view of Proposition 2.8(ii), we can replace $R$ by $R/{\underline{x}}R$ and $C$ by $C/{\underline{x}}C$, and assume that $\depth(R) = 0$. In this case, Lemma 3.1 yields \\
    \centerline{$ \Hom_R(C,E(k)) \otimes_R \Hom_R(C,E(k)) \neq 0$.}
Observe that, $ \big( \Hom_R(C,E(k)) \otimes_R \Hom_R(C,E(k)) \big)^{\vee} \cong \Hom_R(\Hom_R(C,E(k)),C) \in \mathcal{F}_C $. Hence 
$ \Hom_R(C,E(k)) \otimes_R \Hom_R(C,E(k)) \in \mathcal{I}_C $ by 2.6(i), so that $C$ is dualizing by Lemma 3.2.
\end{proof}

\begin{rem}
\emph{After preparing this paper and while we wanted to submit it to "arxiv",  we observed the paper: a criterion for dualizing modules \cite{DNT}. The proof of Theorem 3.6 (resp. Lemma
3.5(ii)) is essentially the same as the proof of Theorem 2.7 (resp. Lemma 2.6) in
\cite{DNT}. It is just dual arguments.}
\end{rem}

The following corollary is an application of Theorem 3.6 in local cohomology. For an $R$-module $M$, the $ i $-th local cohomology module of $M$ with respect to an ideal $ \fa $ of $R$, denoted by $ \H^i_{\fa}(M) $, is defined to be \\
\centerline{$  \H^i_{\fa}(M) = \underset{\underset{n \geq 1} \longrightarrow}  \lim \Ext^i_R(R/ \fa^n , M) $.}
For the basic properties of local cohomology modules, please see the textbook \cite{BS}.

\begin{cor}\label{P}
  Let $(R, \fm)$ be a $d$-dimensional Cohen-Macaulay local ring. Then
  $$\emph\Tor^R_i(\emph\H^d_{\fm}(R),\emph\H^d_{\fm}(R)) \cong \left\lbrace
           \begin{array}{c l}
              \emph\H^d_{\fm}(R)\ \ & \text{ \ \ \ \ \ $i=d$,}\\
              0\ \   & \text{  \ \ $\text{ \ \ $i \neq d$}$.}
           \end{array}
        \right.$$
\end{cor}
\begin{proof}
Note that $ \widehat{R} $ is a $ d $-dimensional complete Cohen-Macaulay ring, and hence admits a canonical module $\omega _{\widehat{R}}$. The $R$-module $ \H^d_{\fm}(R) $ is Artinian by \cite[Theorem 7.1.6]{BS}. Hence $ \Tor^R_i(\H^d_{\fm}(R) , \H^d_{\fm}(R)) $ is Artinian for all $ i \geq 0 $ by \cite[Corollary 3.2]{KLW1}. Hence there are isomorphisms
 \[\begin{array}{rl}
 \Tor^R_i(\H^d_{\fm}(R),\H^d_{\fm}(R)) &\cong \Tor^R_i(\H^d_{\fm}(R),\H^d_{\fm}(R)) \otimes_R \widehat{R}  \\
 &\cong \Tor^{\widehat{R}}_i(\H^d_{\fm}(R) \otimes_R \widehat{R} ,\H^d_{\fm}(R) \otimes_R \widehat{R})\\
 &\cong \Tor^{\widehat{R}}_i \big(\H^d_{\fm\widehat{R}}(\widehat{R}) ,\H^d_{\fm\widehat{R}}(\widehat{R}) \big), \\
 \end{array}\]
 in which the second isomorphism is from \cite[Theorem 2.1.11]{EJ1}, and the last one is the flat base change \cite[Theorem 4.3.2]{BS}. On the other hand, there are isomorphisms
    \[\begin{array}{rl}
    \H^d_{\fm}(R) &\cong \H^d_{\fm}(R) \otimes_R \widehat{R} \\
    &\cong \H^d_{\fm \widehat{R}}(\widehat{R})\\
    &\cong \Hom_{\widehat{R}} \big(\omega _{\widehat{R}}, E_{\widehat{R}}(\widehat{R}/ \fm\widehat{R}) \big), \\
    \end{array}\]
in which the first isomorphism holds because $ \H^d_{\fm}(R) $ is Artinian, the second isomorphism is the flat base change, and the last one is the local duality \cite[Theorem 11.2.8]{BS}. Hence $\H^d_{\fm \widehat{R}}(\widehat{R})$ 
is $\omega _{\widehat{R}}$-injective $ \widehat{R} $-module and then by Theorem 3.5, we have

$$\Tor^{ \widehat{R} }_i(\H^d_{\fm\widehat{R}}(\widehat{R}),\H^d_{\fm\widehat{R}}(\widehat{R})) \cong \left\lbrace
           \begin{array}{c l}
              \H^d_{\fm\widehat{R}}(\widehat{R})\ \ & \text{ \ \ \ \ \ $i=d$,}\\
              0\ \   & \text{  \ \ $\text{ \ \ $i \neq d$}$.}
           \end{array}
        \right.\\$$
Therefore, in view of \cite[Remark 10.2.9]{BS}, we have 
  $$\Tor^R_i(\H^d_{\fm}(R),\H^d_{\fm}(R)) \cong \left\lbrace
  \begin{array}{c l}
  \H^d_{\fm}(R)\ \ & \text{ \ \ \ \ \ $i=d$,}\\
  0\ \   & \text{  \ \ $\text{ \ \ $i \neq d$}$.}
  \end{array}
  \right.$$
\end{proof}

\textbf{Acknowledgments.} We thank the referee for very careful reading of the manuscript and also for his/her useful
suggestions.
 		
\bibliographystyle{amsplain}

\end{document}